\documentclass[a4paper,11pt]{amsart}
\usepackage{amsmath,amssymb,amsthm}
\usepackage{microtype}
\newcommand{\kk}{\Bbbk}
\usepackage{enumitem}
\setlist[enumerate]{font=\normalfont,leftmargin=*,beginpenalty=10000}
\usepackage{xcolor}
\definecolor{citeblue}{RGB}{0,0,128}
\definecolor{refred}{RGB}{128,0,0}
\usepackage[colorlinks=true,citecolor=citeblue,linkcolor=refred,urlcolor=citeblue]{hyperref}
\theoremstyle{plain}
\newtheorem{theorem}{Theorem}[section]
\newtheorem{lemma}[theorem]{Lemma}
\newtheorem{corollary}[theorem]{Corollary}
\newtheorem{mainthm}{Theorem}

\theoremstyle{definition}
\newtheorem{example}[theorem]{Example}
\theoremstyle{remark}
\newtheorem{remark}[theorem]{Remark}
\title{A classification of nearly Gorenstein monomial~ideals}
\author{Sora Miyashita}
\address{Department of Pure and Applied Mathematics, Graduate School of Information Science and Technology, The University of Osaka, Suita, Osaka 565-0871, Japan}
\email{u804642k@ecs.osaka-u.ac.jp}
\date{October 31, 2026}
\begin{document}
\begin{abstract}
We classify nearly Gorenstein monomial ideals over an arbitrary field. In dimension zero, we compute the canonical trace ideal using a graph on the socle monomials. In positive dimension, we prove a dimension bound under unmixedness and a connectedness condition, with weaker assumptions on the canonical trace ideal.
\end{abstract}
\maketitle

\section{Introduction}

Let $R$ be a Noetherian positively graded algebra over a field $\kk$, with $R_0=\kk$ and graded maximal ideal $\mathfrak m_R$. Let $\omega_R$ denote its \textit{graded canonical module}. The \textit{canonical trace ideal} is
\[
\operatorname{tr}_R(\omega_R)=\sum_{\varphi\in\operatorname{Hom}_R(\omega_R,R)}\varphi(\omega_R).
\]
If $R$ is Cohen--Macaulay, this ideal defines its non-Gorenstein locus. The ring is called \textit{nearly Gorenstein} if it is Cohen--Macaulay and $\mathfrak m_R\subseteq\operatorname{tr}_R(\omega_R)$ (see \cite{HHS19}).

Let $S=\kk[x_1,\ldots,x_n]$ with $\deg x_i>0$, and let $I\subseteq S$ be a monomial ideal. Set $R=S/I$, so $\mathfrak m_R=(x_1,\ldots,x_n)R$. We assume $I\subseteq(x_1,\ldots,x_n)^2$, deleting any variables in $I$. When $R$ is Artinian, write $n_i$ for the least integer such that $x_i^{n_i}\in I$, and let $G$ be the graph on $[n]=\{1,\ldots,n\}$ with an edge $\{i,j\}$ whenever $i\ne j$ and $x_ix_j\notin I$. For a connected component $C$ of $G$, set $I_C=I\cap \kk[x_i:i\in C]$.

The \textit{fiber product} $A_1\times_{\kk}\cdots\times_{\kk} A_r$ consists of the tuples in $\prod_i A_i$ with equal constant terms; we allow $r=1$.

Our main result is the following classification.

\begin{mainthm}\label{thm:A}
Suppose that $R$ is Cohen--Macaulay and non-Gorenstein. Then $R$ is nearly Gorenstein if and only if $\dim R\leq2$ and, after renumbering the variables, the following conditions hold:
\begin{enumerate}
\item[(i)] If $\dim R=0$, each $I_C$ belongs to exactly one of the following families:
\begin{enumerate}
\item[(a)] $I_C=(x_i^{n_i}:i\in C)$;
\item[(b)] $|C|\geq2$ and $I_C=(x_i^{n_i}:i\in C)+(\prod_{i\in C}x_i^{n_i-1})$. If $|C|=2$, at least one $n_i$ is greater than $2$;
\item[(c)] $|C|\geq2$ and $I_C=(x_i^{n_i}:i\in C)+(x_i^{n_i-1}x_j:i\in P,\ j\in C\setminus\{i\})$ for a nonempty set $P\subseteq\{i\in C:n_i\geq3\}$. If $|C|=2$, then $|P|=1$ and $n_i\geq3$ for both $i\in C$;
\end{enumerate}
\item[(ii)] If $\dim R=1$, then $R$ is a fiber product over $\kk$ of one-dimensional monomial complete intersections and rings $\kk[x,y,z]/(xy,yz^{a-1},z^a)$, on disjoint sets of variables, with $a\geq3$;
\item[(iii)] If $\dim R=2$, then $I=(x_ix_j:1\leq i<j\leq n,\ j-i\geq2)$ with $n\geq4$.
\end{enumerate}
\end{mainthm}

In (i), the non-Gorenstein hypothesis can be omitted. Case (a) includes singleton components. If $G$ is connected and $R$ is not Gorenstein, only (b) and (c) occur.

With the standard grading, the rings $\kk[x_i:i\in C]/I_C$ in (i) are, respectively, complete intersections, non-Gorenstein \textit{level} rings, and nonlevel rings (see Remark~\ref{rem:artinian-disjoint}). The rings in (iii) are the Stanley--Reisner rings of paths with at least four vertices.

For Artinian local rings, Huneke and Vraciu characterized nearly Gorenstein rings of type two and classified nearly Gorenstein monomial ideals of type three (see \cite[Theorems 4.2 and 5.4]{HV06}). Ficarra classified the nearly Gorenstein monomial ideals of codimension two \cite[Theorem 4.1]{F25}. In the squarefree case, Miyashita and Varbaro \cite{MV25} extended \cite[Theorem 4.3]{M24} and proved that a non-Gorenstein nearly Gorenstein Stanley--Reisner ring has dimension at most two (\cite[Theorem A~(Y)]{MV25}), answering \cite[Question 4.6]{M24}. Theorem~\ref{thm:A} classifies nearly Gorenstein monomial ideals over an arbitrary field.

The dimension bound extends to unmixed rings under the following connectedness assumption.

\begin{mainthm}\label{thm:B}
Let $R=S/I$, where $I\subseteq S$ is a monomial ideal, and let $\Delta=\{F\subseteq[n]:\prod_{i\in F}x_i\notin\sqrt I\}$ be the simplicial complex defined by $\sqrt I$. Suppose that $R$ is unmixed. If $\dim R\geq2$, assume in addition that $\Delta$ is connected. If
\[
\sqrt{\operatorname{tr}_R(\omega_R)}=\mathfrak m_R,\qquad \operatorname{tr}_R(\omega_R)\nsubseteq\mathfrak m_R^2,
\]
then $R$ is Cohen--Macaulay and $\dim R\leq2$.
\end{mainthm}

The vertices of $\Delta$ correspond to the variables that are not nilpotent in $R$. The canonical trace ideal is monomial, so the second condition says that it contains a variable. Both conditions are independent of the positive weights. The hypotheses on $R$ and $\Delta$ hold under Serre's condition $(S_2)$ (Corollary~\ref{cor:s2-bound}). Thus Theorem~\ref{thm:B} gives the dimension bound in Theorem~\ref{thm:A}, since a non-Gorenstein nearly Gorenstein ring has canonical trace ideal $\mathfrak m_R$.

We classify the one-dimensional graded almost Gorenstein monomial rings over an infinite field (Corollary~\ref{cor:almost-one}) and prove that every Cohen--Macaulay monomial ring of even dimension at least four which is locally Gorenstein is Gorenstein (Corollary~\ref{cor:even}). Here \textit{locally Gorenstein} means that $R_{\mathfrak p}$ is Gorenstein for every prime $\mathfrak p\ne\mathfrak m_R$.

For an Artinian monomial ring, we give a formula for its canonical trace ideal and a criterion for a monomial to belong to it in terms of socle monomials (Theorem~\ref{thm:trace}). We apply the formula to Artinian truncations of Stanley--Reisner rings (Corollary~\ref{cor:sr-truncations}) and obtain an explicit description for flag complexes (Corollary~\ref{cor:flag-trace}). We also use it to prove the Artinian classification in Section~\ref{sec:artinian}. Sections~\ref{sec:positive}--\ref{sec:two} treat positive-dimensional monomial rings. The final section proves Theorem~\ref{thm:B} and completes the classification in Theorem~\ref{thm:A}.

\section{Canonical trace ideals and monomial ideals}

Throughout this paper, all rings are commutative and Noetherian. We write $M_j$ for the component of degree $j$ of a graded module $M$, and use the shift convention $[M(a)]_j=M_{j+a}$. For a positively graded $\kk$-algebra $R=\bigoplus_{j\geq0}R_j$ with $R_0=\kk$, we write $\mathfrak m_R=\bigoplus_{j>0}R_j$ for its graded maximal ideal. We set $d=\dim R$. We call $I\subseteq S$ nearly Gorenstein, Cohen--Macaulay, or Gorenstein if $S/I$ has the corresponding property.

The \textit{trace} of an $R$-module $M$ is the ideal $\operatorname{tr}_R(M)=\sum_{\varphi\in\operatorname{Hom}_R(M,R)}\varphi(M)$. The ideal $\operatorname{tr}_R(\omega_R)$ is the \textit{canonical trace ideal} of $R$.

We call $R$ \textit{quasi-Gorenstein} if $\omega_R\cong R(a)$ as graded modules for some $a\in\mathbb Z$. This is equivalent to $\operatorname{tr}_R(\omega_R)=R$ (see \cite[Proposition 3.3]{AG85} and \cite[Remark 2.9]{KM26}). For a Cohen--Macaulay ring, quasi-Gorensteinness is equivalent to Gorensteinness. Thus a non-Gorenstein Cohen--Macaulay ring is nearly Gorenstein if and only if $\operatorname{tr}_R(\omega_R)=\mathfrak m_R$.

A ring is \textit{generically Gorenstein} if its localizations at minimal primes are Gorenstein. A nearly Gorenstein ring is locally Gorenstein by \cite[Corollary 3.4]{AG85}, and is therefore generically Gorenstein when its dimension is positive.

Write $\mathbb N$ for the nonnegative integers, $\mathbf 1=(1,\ldots,1)$, and $\mathbf e_i$ for the $i$th standard basis vector of $\mathbb Z^n$. We give $S=\kk[x_1,\ldots,x_n]$ and $R=S/I$ the $\mathbb Z^n$-grading defined by $\deg x_i=\mathbf e_i$, and use the shift convention $[M(\mathbf a)]_{\mathbf u}=M_{\mathbf u+\mathbf a}$. When comparing degrees of monomials, we use total degree. Levelness, socle degrees, and graded almost Gorensteinness refer to the standard grading.

For a multigraded $R$-module $M$ with finite-dimensional homogeneous components, its \textit{graded $\kk$-dual} is $M^\vee=\bigoplus_{\mathbf u\in\mathbb Z^n}\operatorname{Hom}_{\kk}(M_{-\mathbf u},\kk)$, with $(a\lambda)(m)=\lambda(am)$. It is naturally a submodule of $\operatorname{Hom}_{\kk}(M,\kk)$, and equals that module when $M$ is finite-dimensional.

With this multigrading, $\omega_R=\operatorname{Ext}^{n-d}_S(R,S(-\mathbf 1))$. Giving an element of multidegree $\mathbf u$ the degree $\sum_i u_i\deg x_i$ yields the graded canonical module for the chosen positive degrees. Graded local duality identifies $\omega_R$ with $H^d_{\mathfrak m_R}(R)^\vee$ (see \cite[Theorem 2.2.2]{GW78}).

For a Cohen--Macaulay ring $R$, the \textit{type} $\operatorname{type}(R)$ is the minimal number of generators of $\omega_R$. The \textit{embedding dimension} $\operatorname{embdim}(R)$ is the minimal number of generators of $\mathfrak m_R$, and the \textit{codimension} is $\operatorname{codim}(R)=\operatorname{embdim}(R)-d$. A standard graded Cohen--Macaulay ring is \textit{level} if its canonical module is generated in one total degree.

For $\mathbf u=(u_1,\ldots,u_n)\in\mathbb N^n$, write $\mathbf x^{\mathbf u}=x_1^{u_1}\cdots x_n^{u_n}$. A monomial $\mathbf x^{\mathbf u}$ is \textit{standard} if $\mathbf x^{\mathbf u}\notin I$; we use the same symbol for its residue class. The standard monomials form a $\kk$-basis of $R$. An ideal generated by their residue classes is called a \textit{monomial ideal} of $R$, and has as a basis the standard monomials it contains. The \textit{support} of $\mathbf x^{\mathbf u}$ is $\operatorname{supp}(\mathbf x^{\mathbf u})=\{i\in[n]:u_i>0\}$. Divisibility of monomials is always taken in $S$.

Although trace uses arbitrary $R$-linear maps, every homomorphism $M\to N$ between finitely generated multigraded modules is a finite sum of homogeneous maps: choose homogeneous generators and decompose their images. Thus $\operatorname{tr}_R(M)$ is multigraded. Since each multigraded component of $R=S/I$ has dimension at most one, membership of a nonzero variable $x_i$ in $\operatorname{tr}_R(\omega_R)$ gives a homogeneous map $\varphi:\omega_R\to R$ and a homogeneous element $w\in\omega_R$ with $\varphi(w)=x_i$.

\begin{remark}
Let $R=S/I$, where $I$ is monomial, and assign arbitrary positive degrees to the variables. Since $\omega_R=\operatorname{Ext}^{n-d}_S(R,\omega_S)$ and $\omega_S\cong S$ after forgetting the grading, the underlying $R$-module $\omega_R$ is independent of these degrees. The ideal $\mathfrak m_R=(x_1,\ldots,x_n)R$ is also unchanged. Thus the canonical trace ideal and the nearly Gorenstein property are independent of the weights.
\end{remark}

The \textit{facets} of a simplicial complex $\Delta$ are its inclusion-maximal faces; we write $\mathcal F(\Delta)$ for their set. A complex is \textit{pure} if all its facets have the same cardinality, and its \textit{dimension} is $\max\{|F|-1:F\in\Delta\}$. If its vertex set is $V$, its \textit{Stanley--Reisner ideal} is $I_\Delta=(\prod_{i\in C}x_i:C\subseteq V,\ C\notin\Delta)$, and $\kk[\Delta]=\kk[x_i:i\in V]/I_\Delta$ is its \textit{Stanley--Reisner ring}. The \textit{link} of a face $C$ is $\operatorname{lk}_\Delta C=\{F\in\Delta:F\cap C=\varnothing,\ F\cup C\in\Delta\}$. A vertex contained in every facet is a \textit{cone point}. For a pure complex of dimension $d-1$, its \textit{dual graph} has vertex set $\mathcal F(\Delta)$, with $F,F'$ adjacent when $|F\cap F'|=d-1$. For a face $C$, we write $\mathbf{x}_C=\prod_{i\in C}x_i$.

\section{Dimension zero}\label{sec:artinian}

Let $A=S/I$ be Artinian, where $S=\kk[x_1,\ldots,x_n]$ and $I\subseteq(x_1,\ldots,x_n)^2$ is a monomial ideal. The case $A=\kk$ is clear, so we assume $n\geq1$. The \textit{socle} of $A$ is $\operatorname{Soc}(A)=0:_A\mathfrak m_A$. We have $\omega_A=A^\vee$ and $\operatorname{type}(A)=\dim_{\kk}\operatorname{Soc}(A)$. Thus $A$ is Gorenstein if and only if its socle has dimension one, and is level if and only if its socle is concentrated in one degree. Write $\mathcal S(A)$ for the set of \textit{socle monomials}, namely the standard monomials in $\operatorname{Soc}(A)$. Every standard monomial divides an element of $\mathcal S(A)$, so this set determines $I$. The \textit{socle degree} is $\max\{j:A_j\ne0\}$.

For a monomial $\mathbf{w}$, let $\Gamma_{\mathbf{w}}$ be the graph on $\mathcal S(A)$ with an edge between distinct vertices $\mathbf{v},\mathbf{h}$ if $\mathbf{w}\mid \mathbf{v}\mathbf{h}$. Write $C_{\mathbf{v}}(\mathbf{w})$ for the connected component of $\Gamma_{\mathbf{w}}$ containing $\mathbf{v}$. The product $\mathbf{v}\mathbf{h}$ and divisibility are taken in the polynomial ring; $\mathbf{w}$ need not be standard.

The following formula computes the canonical trace ideal of an Artinian monomial ring.

\begin{theorem}\label{thm:trace}
The canonical trace ideal of $A$ is
\[
\operatorname{tr}_A(\omega_A)=\bigl(\mathbf{u} : \mathbf{u}\text{ is standard, and for some }\mathbf{v}\in\mathcal S(A),\ \mathbf{h}\mid \mathbf{u}\mathbf{v}\text{ for every }\mathbf{h}\in C_{\mathbf{v}}(\mathbf{u}\mathbf{v})\bigr)A.
\]
\end{theorem}

\begin{proof}
By \cite[Corollaries~4.5 and~4.8]{GHHM22}, the canonical trace ideal is the sum of the monomial ideals $J\subseteq A$ whose minimal monomial generators can be paired bijectively with the socle monomials of $J$ so that all paired products are equal in $S$. Here $\operatorname{Soc}(J)=J\cap\operatorname{Soc}(A)$. We apply this description to a standard monomial $\mathbf{u}$.

Suppose that $\mathbf{u}$ is standard and that every vertex of $C_{\mathbf{v}}(\mathbf{u}\mathbf{v})$ divides $\mathbf{u}\mathbf{v}$ for some $\mathbf{v}\in\mathcal S(A)$. Set $\mathbf{w}=\mathbf{u}\mathbf{v}$, $C=C_{\mathbf{v}}(\mathbf{w})$, and $J=(\mathbf{w}/\mathbf{h}:\mathbf{h}\in C)A$. Each generator is standard: $\mathbf{w}/\mathbf{v}=\mathbf{u}$, and if $\mathbf{h}\ne \mathbf{v}$, choose a vertex $\mathbf{k}$ adjacent to $\mathbf{h}$ on a path from $\mathbf{v}$ to $\mathbf{h}$; then $\mathbf{w}/\mathbf{h}\mid \mathbf{k}$.

The socle monomials of $J$ are exactly $C$. Indeed, if $\mathbf{k}\in\mathcal S(A)$ belongs to $J$, then $\mathbf{w}/\mathbf{h}\mid \mathbf{k}$ for some $\mathbf{h}\in C$, so $\mathbf{k}=\mathbf{h}$ or $\mathbf{k}$ is adjacent to $\mathbf{h}$; hence $\mathbf{k}\in C$. Conversely, for $\mathbf{h}\in C$, choose a socle monomial $\mathbf{k}$ divisible by $\mathbf{w}/\mathbf{h}$. Then $\mathbf{k}\in C$ by the preceding argument, and $\mathbf{w}/\mathbf{k}\mid \mathbf{h}$ gives $\mathbf{h}\in J$. The generators $\mathbf{w}/\mathbf{h}$ are minimal because the socle monomials are pairwise incomparable. Pairing $\mathbf{w}/\mathbf{h}$ with $\mathbf{h}$ now shows that $J\subseteq\operatorname{tr}_A(\omega_A)$, and in particular that $\mathbf{u}$ belongs to this ideal.

Conversely, if $\mathbf{u}\in\operatorname{tr}_A(\omega_A)$, choose an ideal $J$ with such a pairing and a minimal generator $\mathbf{a}\mid \mathbf{u}$. Let $\mathbf{v}$ be paired with $\mathbf{a}$, and set $\mathbf{w}_0=\mathbf{a}\mathbf{v}$. Every socle monomial of $J$ divides $\mathbf{w}_0$. Moreover, if $\mathbf{h}$ is one of them and $\mathbf{w}_0\mid \mathbf{h}\mathbf{k}$ for $\mathbf{k}\in\mathcal S(A)$, then $\mathbf{w}_0/\mathbf{h}\mid \mathbf{k}$, so $\mathbf{k}\in J$. Thus all vertices of $C_{\mathbf{v}}(\mathbf{w}_0)$ are socle monomials of $J$. Since $\mathbf{w}_0\mid \mathbf{u}\mathbf{v}$, we have $C_{\mathbf{v}}(\mathbf{u}\mathbf{v})\subseteq C_{\mathbf{v}}(\mathbf{w}_0)$, and all vertices of $C_{\mathbf{v}}(\mathbf{u}\mathbf{v})$ divide $\mathbf{u}\mathbf{v}$. This proves the claimed criterion for each standard monomial, and hence the formula.
\end{proof}

We give two applications of Theorem~\ref{thm:trace} to Artinian truncations of Stanley--Reisner rings.

For a facet $F$ and a face $E$ of a simplicial complex $\Delta$, let $\Gamma_{F,E}$ be the graph on the facets containing $F\cap E$, with distinct $H,L$ adjacent when $F\cup E\subseteq H\cup L$. We call the pair $(F,E)$ \textit{admissible} if every facet in the connected component of $F$ is contained in $F\cup E$.

\begin{corollary}\label{cor:sr-truncations}
Let $\Delta$ be a simplicial complex with vertex set $[n]$, let $n_1,\ldots,n_n\geq2$ be integers, and set $A=\kk[x_1,\ldots,x_n]/(I_\Delta+(x_1^{n_1},\ldots,x_n^{n_n}))$. Then
\[
\operatorname{tr}_A(\omega_A)=\left(\mathbf{x}_{F\cap E}\prod_{j\in E\setminus F}x_j^{n_j-1}:(F,E)\text{ is admissible}\right)A,
\]
where $\mathbf{x}_{F\cap E}=\prod_{i\in F\cap E}x_i$.
\end{corollary}

\begin{proof}
The socle monomials of $A$ are $\mathbf{v}_H=\prod_{i\in H}x_i^{n_i-1}$ for $H\in\mathcal F(\Delta)$: a standard monomial is maximal under divisibility exactly when its support is a facet and its exponents on that facet are maximal.

Let $\mathbf{u}$ be a standard monomial and set $E=\operatorname{supp}(\mathbf{u})$. For any facet $F$, comparison of exponents gives
\[
\mathbf{u}\mathbf{v}_F\mid\mathbf{v}_H\mathbf{v}_L\quad\Longleftrightarrow\quad F\cap E\subseteq H\cap L\ \text{ and }\ F\cup E\subseteq H\cup L.
\]
Thus the component of $\mathbf{v}_F$ in $\Gamma_{\mathbf{u}\mathbf{v}_F}$ corresponds to the component of $F$ in $\Gamma_{F,E}$; facets not containing $F\cap E$ are isolated in the former graph.

Suppose that $\mathbf{u}$ belongs to the canonical trace ideal, and choose $F$ by Theorem~\ref{thm:trace}. Every facet $H$ containing $E$ is either $F$ or adjacent to $F$ in $\Gamma_{F,E}$. Hence $\mathbf{v}_H\mid\mathbf{u}\mathbf{v}_F$, which forces the exponent of $x_j$ in $\mathbf{u}$ to be $n_j-1$ for every $j\in E\setminus F$. The same divisibility for every facet $H$ in the component gives $H\subseteq F\cup E$. Thus $(F,E)$ is admissible, and the corresponding displayed generator divides $\mathbf{u}$.

Conversely, let $(F,E)$ be admissible and set $\mathbf{u}=\mathbf{x}_{F\cap E}\prod_{j\in E\setminus F}x_j^{n_j-1}$. This monomial is standard and has support $E$. For every facet $H$ in the component of $F$, admissibility and the exponents of $\mathbf{u}$ give $\mathbf{v}_H\mid\mathbf{u}\mathbf{v}_F$. By Theorem~\ref{thm:trace}, $\mathbf{u}$ belongs to the canonical trace ideal. This proves the formula.
\end{proof}

\begin{example}\label{ex:triangle}
Let $\Delta$ be the boundary of a triangle, with facets $\{1,2\}$, $\{2,3\}$, and $\{1,3\}$. For integers $n_1,n_2,n_3\geq2$, set $A=\kk[x_1,x_2,x_3]/(x_1^{n_1},x_2^{n_2},x_3^{n_3},x_1x_2x_3)$.

For $E=\{1\}$ and $F=\{2,3\}$, the graph $\Gamma_{F,E}$ is a triangle on all three facets: $F\cap E=\varnothing$, and any two distinct facets have union $\{1,2,3\}=F\cup E$. Thus $(F,E)$ is admissible, giving the generator $x_1^{n_1-1}$ in Corollary~\ref{cor:sr-truncations}. If instead $F$ contains $1$, the graph has the two adjacent vertices $\{1,2\}$ and $\{1,3\}$, whose union is not contained in $F\cup E=F$. Such a pair is not admissible. By symmetry, a singleton $E=\{i\}$ gives exactly the generator $x_i^{n_i-1}$.

For an edge $E=\{i,j\}$, the pair $(E,E)$ is admissible and gives $x_ix_j$; every other generator associated to $E$ is divisible by $x_ix_j$. For $E=\varnothing$, the component of $F$ contains all three facets and is not contained in $F$, so no such pair is admissible. Hence Corollary~\ref{cor:sr-truncations} gives
\[
\operatorname{tr}_A(\omega_A)=(x_1^{n_1-1},x_2^{n_2-1},x_3^{n_3-1},x_1x_2,x_1x_3,x_2x_3)A.
\]
In particular, $\operatorname{tr}_A(\omega_A)=\mathfrak m_A$ if and only if $n_1=n_2=n_3=2$. This special case also appears in \cite[Example 6.2~(g)]{GHHM22}.
\end{example}

A simplicial complex $\Delta$ is \textit{flag} if its Stanley--Reisner ideal $I_\Delta$ is generated by squarefree monomials of degree two. For $n_1=\cdots=n_n=2$, the following formula was proved in \cite[Theorem 6.1]{GHHM22}.

\begin{corollary}\label{cor:flag-trace}
Let $\Delta$ be a flag simplicial complex with vertex set $[n]$, and let $A=\kk[x_1,\ldots,x_n]/(I_\Delta+(x_1^{n_1},\ldots,x_n^{n_n}))$, where $n_1,\ldots,n_n\geq2$ are integers. Then
\[
\operatorname{tr}_A(\omega_A)=\bigl(\mathbf{x}_U:U\in\Delta\text{ is contained in exactly one facet}\bigr)A,
\]
where $\mathbf{x}_U=\prod_{i\in U}x_i$ and $\mathbf{x}_\varnothing=1$. These monomial generators do not depend on $n_1,\ldots,n_n$.
\end{corollary}

\begin{proof}
Set $B=\kk[x_1,\ldots,x_n]/(I_\Delta+(x_1^2,\ldots,x_n^2))$. Admissibility depends only on $\Delta$, so Corollary~\ref{cor:sr-truncations} gives $\mathbf{x}_E\in\operatorname{tr}_B(\omega_B)$ for every admissible pair $(F,E)$. By \cite[Theorem 6.1]{GHHM22}, $E$ contains a face lying in a unique facet, so $E$ itself lies in that facet alone. Each generator in Corollary~\ref{cor:sr-truncations} is divisible by $\mathbf{x}_E$, and hence belongs to the ideal on the right.

Conversely, suppose that $E$ is contained in a unique facet $F$. Since $E\subseteq F$, the graph $\Gamma_{F,E}$ has only the vertex $F$. Thus $(F,E)$ is admissible, and its generator in Corollary~\ref{cor:sr-truncations} is $\mathbf{x}_E$. This proves the reverse inclusion.
\end{proof}

\begin{example}
Let $A=\kk[x,y]/(x^2,xy,y^3)$. The simplicial complex with facets $\{1\}$ and $\{2\}$ is flag and has Stanley--Reisner ideal $(xy)$. Its faces contained in exactly one facet are precisely $\{1\}$ and $\{2\}$. Hence Corollary~\ref{cor:flag-trace}, with exponents $2$ and $3$, gives $\operatorname{tr}_A(\omega_A)=(x,y)A$.
\end{example}

We return to an arbitrary Artinian monomial ring $A=\kk[x_1,\ldots,x_n]/I$, with $I\subseteq(x_1,\ldots,x_n)^2$, and prove Theorem~\ref{thm:A}~(i). For each $i$, let $n_i$ be the least integer with $x_i^{n_i}\in I$. Recall that $G$ has vertex set $[n]$, with $i$ and $j$ adjacent when $i\ne j$ and $x_ix_j\notin I$. For each connected component $C$, set $I_C=I\cap\kk[x_i:i\in C]$ and $A_C=\kk[x_i:i\in C]/I_C$. Variables in distinct components multiply to zero, so $A$ is the fiber product of the $A_C$ over $\kk$ and $\mathfrak m_A=\bigoplus_C\mathfrak m_{A_C}$. We identify $\mathfrak m_{A_C}$ with the ideal spanned by the positive-degree standard monomials supported on $C$.

The following consequence of Theorem~\ref{thm:trace} restricts the standard monomials whenever a variable belongs to the canonical trace ideal.

\begin{lemma}\label{lem:trace-variables}
If $x_i\in\operatorname{tr}_A(\omega_A)$, there is a socle monomial $\mathbf{v}_i$ with the following properties:
\begin{enumerate}
\item Every standard multiple of $x_i$ divides $x_i\mathbf{v}_i$;
\item The neighbors of $i$ in $G$ are pairwise adjacent.
\end{enumerate}
\end{lemma}

\begin{proof}
By Theorem~\ref{thm:trace}, choose a socle monomial $\mathbf{v}_i$ such that every vertex in its connected component in $\Gamma_{x_i\mathbf{v}_i}$ divides $x_i\mathbf{v}_i$. If a socle monomial $\mathbf{h}$ is divisible by $x_i$, then $x_i\mathbf{v}_i\mid \mathbf{v}_i\mathbf{h}$. Hence either $\mathbf{h}=\mathbf{v}_i$ or $\mathbf{h}$ is adjacent to $\mathbf{v}_i$ in $\Gamma_{x_i\mathbf{v}_i}$. In either case, $\mathbf{h}\mid x_i\mathbf{v}_i$. Every standard multiple of $x_i$ divides such an $\mathbf{h}$, proving (1). If $j,k$ are distinct neighbors of $i$, then (1) gives $x_j\mid \mathbf{v}_i$ and $x_k\mid \mathbf{v}_i$. Thus $x_jx_k$ is standard, proving (2).
\end{proof}

\begin{proof}[Proof of Theorem~\ref{thm:A}~(i)]
Each $A_C$ is a nontrivial Artinian graded algebra. By \cite[Theorem 3.16~(1),~(2)]{KM26} and induction on the number of components, $A$ is nearly Gorenstein if and only if every $A_C$ is nearly Gorenstein.

We may now assume that $G$ is connected. If $A$ is Gorenstein, it has a unique socle monomial, and its standard monomials are exactly the divisors of that monomial. Thus $I$ is generated by pure powers, giving (a). In particular, the case $n=1$ is settled.

We first prove the necessity when $n\geq2$ and $A$ is not Gorenstein. For each $i$, choose a socle monomial $\mathbf{v}_i$ as in Lemma~\ref{lem:trace-variables}, so that every standard multiple of $x_i$ divides $x_i\mathbf{v}_i$. By part (2) of the lemma, the neighbors of each vertex of $G$ are pairwise adjacent. As $G$ is connected, it is complete: otherwise a shortest path between two nonadjacent vertices could be shortened. Part (1) then gives $x_j\mid \mathbf{v}_i$ whenever $j\ne i$.

Set $\mathbf{w}=\operatorname{lcm}(\mathbf{v}_1,\ldots,\mathbf{v}_n)$. For $i\ne j$, the monomial $\mathbf{v}_j$ is divisible by $x_i$, so Lemma~\ref{lem:trace-variables}~(1) gives $\mathbf{v}_j\mid x_i\mathbf{v}_i$. Hence $\mathbf{v}_i\mid \mathbf{w}\mid x_i\mathbf{v}_i$. If $\mathbf{v}_i=\mathbf{w}$ for some $i$, then all the $\mathbf{v}_j$ equal $\mathbf{w}$, since distinct socle monomials do not divide one another. Otherwise $\mathbf{w}=x_i\mathbf{v}_i$ for every $i$.

Suppose first that $x_i\mathbf{v}_i=\mathbf{w}$ for every $i$. Every positive-degree standard monomial divides $\mathbf{w}$ by Lemma~\ref{lem:trace-variables}~(1). Conversely, every proper divisor of $\mathbf{w}$ divides some $\mathbf{w}/x_i=\mathbf{v}_i$ and is therefore standard. Since $\mathbf{w}=x_i\mathbf{v}_i\in I$, the standard monomials are precisely the proper divisors of $\mathbf{w}$. Hence $\mathbf{w}=\prod_{i=1}^n x_i^{n_i-1}$ and $I=(x_1^{n_1},\ldots,x_n^{n_n})+(\mathbf{w})$, which is (b). If $n=2$, connectedness excludes $\mathbf{w}=x_1x_2$, so at least one $n_i$ is greater than two.

Suppose now that $\mathbf{v}_i=\mathbf{u}=\prod_{i=1}^n x_i^{a_i}$ for every $i$; each $a_i\geq1$. A standard monomial involving distinct variables $x_i,x_j$ divides both $x_i\mathbf{u}$ and $x_j\mathbf{u}$, and hence divides $\mathbf{u}$. A standard pure power of $x_i$ has exponent at most $a_i+1$. Thus the set of standard monomials is $\{\mathbf{t}:\mathbf{t}\mid\mathbf{u}\}\cup\{x_i^{a_i+1}:i\in P\}$ for some $P\subseteq[n]$. Since $A$ is not Gorenstein, $P$ is nonempty.

If $n=2$ and $P=\{1,2\}$, write $\mathbf{u}=x^ay^b$. Then $I=(x^{a+2},x^{a+1}y,xy^{b+1},y^{b+2})$ has four minimal generators. This is impossible, since an Artinian nearly Gorenstein monomial ideal in two variables has at most three minimal generators by \cite[Theorem 4.1]{F25}. Thus, when $n=2$, we have $P=\{i\}$. If the exponent $a_j$ of the other variable is one, this set consists of the proper divisors of $x_i\mathbf{u}$, and we are again in (b). In every other case, set $n_i=a_i+2$ for $i\in P$ and $n_i=a_i+1$ otherwise. The standard monomials then give exactly the ideal in (c), with the stated restrictions.

We now prove the sufficiency. The ideals in (a) are Gorenstein. In (b), $A=D/\operatorname{Soc}(D)$ for $D=S/(x_i^{n_i}:1\leq i\leq n)$. By \cite[Theorem 4.11]{GHHM22}, its canonical trace ideal is $\mathfrak m_A$.

In (c), set $a_i=n_i-2$ on $P$ and $a_i=n_i-1$ outside $P$. The socle monomials are $\mathbf{u}=\prod_i x_i^{a_i}$ and $\mathbf{p}_j=x_j^{a_j+1}$ for $j\in P$. In $\Gamma_{x_i\mathbf{u}}$, the only possible neighbor of $\mathbf{u}$ is $\mathbf{p}_i$, and it is a neighbor precisely when $i\in P$. If $n\geq3$, distinct pure-power vertices cannot be adjacent, since their product omits a variable occurring in $x_i\mathbf{u}$. If $n=2$, there is at most one pure-power vertex. Thus $C_{\mathbf{u}}(x_i\mathbf{u})$ is $\{\mathbf{u}\}$ or $\{\mathbf{u},\mathbf{p}_i\}$, and all its vertices divide $x_i\mathbf{u}$. Theorem~\ref{thm:trace} shows that $A$ is nearly Gorenstein.

The families are disjoint by Remark~\ref{rem:artinian-disjoint} below.
\end{proof}

\begin{remark}\label{rem:artinian-disjoint}
The families in Theorem~\ref{thm:A}~(i) are distinguished by their socles. The rings in (a) are Gorenstein, whereas those in (b) are non-Gorenstein and level: their socle monomials are $\mathbf{w}/x_i$, where $\mathbf{w}=\prod_i x_i^{n_i-1}$. The rings in (c) are not level. Indeed, with $a_i=n_i-2$ for $i\in P$ and $a_i=n_i-1$ otherwise, their socle monomials are $\mathbf{u}=\prod_i x_i^{a_i}$ and $x_i^{a_i+1}$ for $i\in P$. For each $i\in P$, we have $\deg \mathbf{u}-(a_i+1)=\sum_{j\ne i}a_j-1>0$ by the restrictions in (c).
\end{remark}

\begin{example}\label{ex:artinian-families}
Let $A=\kk[x,y]/(x^3,y^2,x^2y)$ and $B=\kk[x,y]/(x^3,y^3,x^2y)$. Their defining ideals belong to cases (b) and (c) of Theorem~\ref{thm:A}~(i), respectively. The socle monomials of $A$ are $x^2$ and $xy$, whereas those of $B$ are $x^2$ and $xy^2$. Thus $A$ is level and $B$ is not.
\end{example}

\begin{example}\label{ex:triangle-classification}
Consider again the ring $A=\kk[x_1,x_2,x_3]/I$ of Example~\ref{ex:triangle}, where $I=(x_1^{n_1},x_2^{n_2},x_3^{n_3},x_1x_2x_3)$ and $n_1,n_2,n_3\geq2$. Its graph $G$ is the complete graph on three vertices. The mixed minimal generator $x_1x_2x_3$ excludes case (a) of Theorem~\ref{thm:A}~(i). It also excludes case (c), whose mixed minimal generators involve only two variables. Thus $A$ is nearly Gorenstein if and only if $I$ belongs to case (b), which requires $x_1x_2x_3=x_1^{n_1-1}x_2^{n_2-1}x_3^{n_3-1}$. This holds precisely when $n_1=n_2=n_3=2$, in agreement with the computation in Example~\ref{ex:triangle}.
\end{example}

\section{Primary components and substitutions by powers}\label{sec:positive}

Let $R=S/I$ have dimension $d\geq1$, where $I$ is a monomial ideal. We say that $R$ is \textit{unmixed} if all associated primes of $R$ have dimension $d$. Write $\sqrt I=\bigcap_{F\in\mathcal F}P_F$ irredundantly, where $P_F=(x_j:j\notin F)$. The sets in $\mathcal F$ are the \textit{facets}, or inclusion-maximal faces, of the simplicial complex $\Delta$ corresponding to $\sqrt I$; thus $S/\sqrt I\cong\kk[\Delta]$, after deleting the variables in $\sqrt I$. If $R$ is unmixed, then $|F|=d$ for every $F\in\mathcal F$.

For a face $C\in\Delta$, set $\mathbf{x}_C=\prod_{i\in C}x_i$. Let $R(C)$ be the $\kk$-algebra on the variables $x_j$, $j\notin C$, defined by the monomials obtained from the generators of $I$ by setting $x_i=1$ for $i\in C$. Then $R_{\mathbf{x}_C}\cong R(C)[x_i^{\pm1}:i\in C]$. The minimal primes of $R(C)$ correspond to the facets containing $C$, and its reduced ring is the Stanley--Reisner ring of $\operatorname{lk}_\Delta C$. If $R$ is unmixed, then $\dim R(C)=d-|C|$.

We shall use two results on Gorenstein monomial ideals. By \cite[Corollary 4.3]{BH95}, every such ideal is obtained from its radical by replacing each variable with a fixed positive power. If $\dim S/I\leq1$, a Gorenstein monomial ideal $I$ is a complete intersection by \cite[Corollary 4.4]{BH95}.

We first explain the effect of replacing variables by powers. For $\mathbf a=(a_1,\ldots,a_n)\in\mathbb N^n$ with $a_i\geq1$, define $\varphi_{\mathbf a}:S\to S$ by $\varphi_{\mathbf a}(x_i)=x_i^{a_i}$. Every monomial can be written uniquely as $\varphi_{\mathbf a}(\mathbf x^{\mathbf v})\mathbf x^{\mathbf u}$ with $0\leq u_i<a_i$. Hence $S$ is free over $\varphi_{\mathbf a}(S)$ with basis $\{\mathbf x^{\mathbf u}:0\leq u_i<a_i\}$ (see also \cite[Remark 2.6]{CV20}). For any monomial ideal $J$, the induced map $S/J\to S/\varphi_{\mathbf a}(J)S$ is therefore finite and faithfully flat. Its fibres are zero-dimensional complete intersections, so Cohen--Macaulayness and Gorensteinness pass in both directions.

Set $A=S/J$ and $A_{\mathbf a}=S/\varphi_{\mathbf a}(J)S$. At the graded maximal ideals, this is a flat local map with closed fibre $\kk[x_1,\ldots,x_n]/(x_1^{a_1},\ldots,x_n^{a_n})$. Since this fibre is Gorenstein, canonical modules commute with this local base change (\cite[Lemma 7.4]{H17}). Trace also commutes with flat base change (\cite[Proposition 2.8~(viii)]{L17}), so
\begin{equation}\label{eq:power-trace}
\operatorname{tr}_{A_{\mathbf a}}(\omega_{A_{\mathbf a}})=\varphi_{\mathbf a}\bigl(\operatorname{tr}_A(\omega_A)\bigr)A_{\mathbf a}.
\end{equation}
Indeed, this equality holds after localization at $\mathfrak m_{A_{\mathbf a}}$, and hence in $A_{\mathbf a}$, since both sides are monomial ideals. Here the map sends the class of $x_i$ in $A$ to the class of $x_i^{a_i}$ in $A_{\mathbf a}$. In particular, the map from $S/\sqrt I$ to $S/I$ used below is induced by this substitution of powers.

The next lemma describes how the exponents in the primary components of $I$ agree. Recall that two facets $F,F'$ are adjacent in the dual graph when $|F\cap F'|=d-1$.

\begin{lemma}\label{lem:components}
Let $R=S/I$ be unmixed of dimension $d\geq1$. Suppose that $R_{\mathfrak p}$ is quasi-Gorenstein for every prime $\mathfrak p\ne\mathfrak m_R$. Then:
\begin{enumerate}
\item $I=\bigcap_{F\in\mathcal F}(x_j^{a_{F,j}}:j\notin F)$ for suitable integers $a_{F,j}\geq1$;
\item If $d\geq2$ and $F,F'\in\mathcal F$ are adjacent, then $a_{F,l}=a_{F',l}$ for every $l\notin F\cup F'$;
\item If $d\geq2$ and, for each $j\in[n]$, the subgraph induced by the facets omitting $j$ is connected or empty, then $I=\varphi_{\mathbf a}(\sqrt I)S$ for some $a_i\geq1$;
\item If $R$ is not quasi-Gorenstein, then $\Delta$ has no cone point.
\end{enumerate}
\end{lemma}

\begin{proof}
Let $C$ be a nonempty face. The canonical trace ideal becomes the unit ideal in $R_{\mathbf{x}_C}$. Since $R_{\mathbf{x}_C}\cong R(C)[x_i^{\pm1}:i\in C]$ is faithfully flat over $R(C)$, base change of trace shows that $R(C)$ is quasi-Gorenstein. A quasi-Gorenstein ring satisfies $(S_2)$, and is therefore Gorenstein in dimension at most two (see \cite[Remark 4.7]{KM26}).

Since $I$ is unmixed, its primary components are $Q_F=IS_{\mathbf{x}_F}\cap S$, where $F\in\mathcal F$ (see \cite[Section 1.3.2]{HH11}). The ring $R(F)$ is Artinian and Gorenstein. By \cite[Corollary 4.3]{BH95}, its defining ideal is generated by pure powers, and the same is true of $Q_F$. This proves (1).

Suppose that $d\geq2$ and that $F,F'$ are adjacent. The face $F\cap F'$ is nonempty, and $R(F\cap F')$ is one-dimensional and Gorenstein. Applying the same result to this ring, the exponent of any variable in its primary components is independent of the component in which it occurs. This gives $a_{F,l}=a_{F',l}$ for $l\notin F\cup F'$, proving (2).

For (3), fix a variable $x_j$. Along any path of adjacent facets that all omit $j$, part (2) shows that its exponent remains unchanged. Thus all $a_{F,j}$ with $j\notin F$ have a common value; denote it by $a_j$. If every facet contains $j$, set $a_j=1$. We then have $Q_F=\varphi_{\mathbf a}(P_F)S$ for every $F$. Flatness preserves finite intersections, so
\[
I=\bigcap_{F\in\mathcal F}\varphi_{\mathbf a}(P_F)S=\varphi_{\mathbf a}\!\left(\bigcap_{F\in\mathcal F}P_F\right)S=\varphi_{\mathbf a}(\sqrt I)S.
\]
Thus one exponent suffices for each variable in all primary components.

Finally, if $i$ is a cone point, none of the generators in (1) involves $x_i$, so $R=R(\{i\})[x_i]$. The first paragraph shows that $R(\{i\})$ is quasi-Gorenstein. Its polynomial extension is also quasi-Gorenstein, proving (4).
\end{proof}

\section{Dimension one}\label{sec:one}

Let $R=S/I$ be Cohen--Macaulay of dimension one, where $I\subseteq(x_1,\ldots,x_n)^2$ is a monomial ideal. For $a\geq3$, set $E_a=\kk[x,y,z]/(xy,yz^{a-1},z^a)$. By \cite[Theorem 5.2]{EGI21}, a fiber product of one-dimensional, generically Gorenstein Cohen--Macaulay rings is nearly Gorenstein if and only if each factor is nearly Gorenstein (see also \cite[Remark 3.18]{KM26} for the graded case).

As in the Artinian case, the connected components of the graph with edges $\{i,j\}$ for $i\ne j$ and $x_ix_j\ne0$ give a fiber product decomposition of $R$ on disjoint sets of variables. Each factor is Cohen--Macaulay of dimension one. Indeed, otherwise it would have a nonzero socle element of positive degree, and this element would also belong to $\operatorname{Soc}(R)$, contrary to $\operatorname{depth}R=1$. The localizations of a factor at its minimal primes agree with the corresponding localizations of $R$. Thus, if $R$ is nearly Gorenstein, each factor is generically Gorenstein and hence nearly Gorenstein by the preceding result. We may therefore restrict attention to rings whose graph is connected.

We shall show that, in the nearly Gorenstein case, connectedness forces $R$ to have at most two minimal primes. The following lemma is the main step. For each variable $x_l$, the set $V_l$ consists of the components on which $x_l$ is nonzero.

\begin{lemma}\label{lem:one-components}
Suppose that the irredundant primary decomposition of $I$ is $I=\bigcap_{i\in V}Q_i$, where $V\subseteq[n]$ and $Q_i=(x_j^{a_{ij}}:j\ne i)$ with $a_{ij}\geq1$. Set $B_i=R_{x_i}$, $T=\{l\in[n]:x_l\in\operatorname{tr}_R(\omega_R)\}$, and $V_l=\{i\in V:x_l/1\ne0\text{ in }B_i\}$. Then:
\begin{enumerate}
\item $1\leq |V_l|\leq2$ for every $l\in T$;
\item If $l\in T$ and $V_l=\{p,q\}$, then $a_{pj}=a_{qj}$ for every $j\notin\{p,q,l\}$;
\item Distinct two-element sets among the $V_l$, $l\in T$, are disjoint.
\end{enumerate}
\end{lemma}

\begin{proof}
For $i\in V$, localizing the primary decomposition gives $B_i\cong(S/Q_i)_{x_i}$, and the kernel of $R\to B_i$ is $Q_i/I$. Since $I=\bigcap_{i\in V}Q_i$, the diagonal map $R\to\bigoplus_{i\in V}B_i$ is injective. Every variable outside $V$ is nilpotent in $R$, so $(x_i:i\in V)R$ has radical $\mathfrak m_R$. For distinct $i,j\in V$, the product $x_ix_j$ is also nilpotent, and hence $R_{x_ix_j}=0$. Thus the \v{C}ech complex on the variables $x_i$, $i\in V$, has only the terms $R$ and $\bigoplus_{i\in V}B_i$. Its first cohomology is $H^1_{\mathfrak m_R}(R)$, giving an exact sequence $0\to R\to\bigoplus_{i\in V}B_i\to H^1_{\mathfrak m_R}(R)\to0$. Taking graded duals and applying graded local duality, we obtain
\[
0\longrightarrow\omega_R\longrightarrow\bigoplus_{i\in V}B_i^\vee\longrightarrow R^\vee\longrightarrow0.
\]
Multiplication in the middle term acts coordinatewise. In each degree $-\mathbf u$, the kernel is spanned by vectors supported on at most two coordinates. To see this, suppose first that some coordinate of $\mathbf u$ is negative. Then $R_{\mathbf u}=0$ and at most one $[B_i]_{\mathbf u}$ is nonzero. If $\mathbf u\in\mathbb N^n$, the nonzero $[B_i]_{\mathbf u}$ have basis $\mathbf x^{\mathbf u}$ and the map from $R_{\mathbf u}$ is diagonal. The dual kernel is therefore spanned by differences of two coordinate vectors.

Fix $l\in T$. Since $R_{\mathbf e_l}=\kk x_l$, there is a homogeneous map that sends one of these spanning vectors $w$ to $x_l$, after rescaling $w$. Write $w=(w_i)_{i\in V}$, let $-\mathbf u$ be its degree, and set $P=\{i\in V:w_i\ne0\}$.

For $i\in P$, the coordinate $w_i$ is a nonzero scalar multiple of the functional dual to the Laurent monomial $\mathbf x^{\mathbf u}$ in $B_i$. Here $u_i\in\mathbb Z$ and $0\leq u_j<a_{ij}$ for $j\ne i$. Multiplication by $x_j$, for $j\ne i$, lowers the exponent $u_j$ by one in the dual basis and gives zero when $u_j=0$, whereas $x_i$ acts invertibly. Thus a monomial $\prod_jx_j^{b_j}$ annihilates $w_i$ exactly when $b_j>u_j$ for some $j\ne i$. The other monomials send $w_i$ to nonzero elements of distinct multidegrees, so $\operatorname{Ann}_S(w_i)=(x_j^{u_j+1}:j\ne i)$. Since multiplication is coordinatewise, $\operatorname{Ann}_S(w)$ is the intersection of these ideals. Moreover, the homomorphism sending $w$ to $x_l$ gives $\operatorname{Ann}_S(w)\subseteq I:x_l$. We obtain
\[
I\subseteq\operatorname{Ann}_S(w)=\bigcap_{i\in P}(x_j^{u_j+1}:j\ne i)\subseteq I:x_l.
\]
The last inclusion implies $V_l\subseteq P$. Indeed, if $i\in V_l\setminus P$, localization at $x_i$ makes $\operatorname{Ann}_S(w)$ the unit ideal, whereas $(I:x_l)S_{x_i}$ remains proper. Since $x_l\ne0$, we also have $V_l\ne\varnothing$. Thus $1\leq|V_l|\leq|P|\leq2$, proving (1).

If $V_l=\{p,q\}$, then $P=\{p,q\}$. Localize the displayed containments at $x_p$ and $x_q$. Taking the colon by $x_l$ leaves the pure-power exponents of the other variables unchanged, so $a_{pj}=u_j+1=a_{qj}$ for $j\notin\{p,q,l\}$. This proves (2). We shall also use the consequence that $p\in V_j$ if and only if $q\in V_j$ for such $j$.

Finally, suppose that distinct two-element sets $V_l=\{p,q\}$ and $V_h$, with $l,h\in T$, intersect. If $h\notin V_l$, the preceding equivalence and the nonempty intersection force $V_h=V_l$, a contradiction. Thus $h\in V_l$; by symmetry, $l\in V_h$. Thus $l,h\in V$, so $l\in V_l$ and $h\in V_h$ as well. Both sets then equal $\{l,h\}$, a contradiction. This proves (3).
\end{proof}

\begin{proof}[Proof of Theorem~\ref{thm:A}~(ii)]
We first prove the sufficiency. Complete intersections are Gorenstein, while $E_a$ is nearly Gorenstein by \cite[Theorem 4.1~(d)]{F25}. The fiber product is then nearly Gorenstein by \cite[Theorem 5.2]{EGI21}.

We now prove the necessity. By the reduction at the beginning of this section, we may assume that $R$ is nearly Gorenstein and its graph is connected. If $R$ is Gorenstein, it is a monomial complete intersection by \cite[Corollary 4.4]{BH95}. We may therefore assume that $R$ is not Gorenstein. Lemma~\ref{lem:components}~(1) gives the primary decomposition required in Lemma~\ref{lem:one-components}. We use the notation of that lemma, with $T=[n]$.

For distinct $l,h$, the primary decomposition shows that $x_lx_h\ne0$ if and only if $V_l\cap V_h\ne\varnothing$. By Lemma~\ref{lem:one-components}~(1),~(3), each $V_l$ has one or two elements, and distinct two-element sets are disjoint. These two-element sets, together with the singleton sets disjoint from them, partition $V$. Variables whose sets $V_l$ lie in different parts are not adjacent. Connectedness therefore forces this partition to have only one part, and hence $|V|\leq2$. The case $|V|=1$ would make $I$ generated by pure powers and $R$ Gorenstein. Thus $V=\{p,q\}$, and connectedness also gives an index $l$ with $V_l=V$.

We next show that there are only three variables. If $h\notin\{p,q,l\}$, Lemma~\ref{lem:one-components}~(2) gives $a_{ph}=a_{qh}=c\geq2$. Hence $R\cong B\otimes_{\kk}\kk[x_h]/(x_h^c)$ for a one-dimensional Cohen--Macaulay ring $B$. Since $R$ is nearly Gorenstein, \cite[Corollary 4.3]{HHS19} would then make it Gorenstein, a contradiction. Thus $n\leq3$. If $n=2$, the ideal $I$ is unmixed of height one and hence principal, again making $R$ Gorenstein. Therefore $n=3$.

Among the ideals in \cite[Theorem 4.1]{F25}, connectedness excludes case~(c) and the case $b=1$ in~(d). The remaining non-Gorenstein ideals give precisely $E_a$ with $a\geq3$.
\end{proof}

We now use the standard grading to discuss almost Gorenstein rings. For a Cohen--Macaulay standard graded ring $R$, write its Hilbert series as $H_R(t)=h_R(t)/(1-t)^{\dim R}$; the numerator $h_R(t)$ is the \textit{$h$-polynomial} of $R$. The ring has \textit{minimal multiplicity} if $h_R(1)=\operatorname{codim} R+1$. If $c=\operatorname{codim}R>0$, this is equivalent to $h_R(t)=1+ct$. In this case $R$ is level of type $c$. Indeed, after extending the field, an Artinian reduction by linear forms has square-zero maximal ideal, which is also its socle.

In dimension one, the \textit{$a$-invariant} is $a(R)=-\min\{j:[\omega_R]_j\ne 0\}=\deg h_R-1$. The ring $R$ is \textit{almost Gorenstein} if there is a degree-zero injection $R\to\omega_R(-a(R))$ whose cokernel is annihilated by $\mathfrak m_R$ (see \cite[Definition 8.1]{GTT15}, and \cite{GMP13} for the one-dimensional local theory).

We record the relevant invariants of $E_a$. The element $x+y$ is a nonzerodivisor, and $E_a/(x+y)\cong \kk[y,z]/(y^2,yz^{a-1},z^a)$ has socle basis $yz^{a-2},z^{a-1}$. Thus $E_a$ is level of type two and $a(E_a)=a-2>0$.

Let $R=A_1\times_{\kk}\cdots\times_{\kk}A_r$, where $r\geq2$ and each $A_i$ is Cohen--Macaulay of dimension one. Applying graded local duality to $0\to R\to\bigoplus_iA_i\to\kk^{r-1}\to0$ gives
\[
0\longrightarrow\bigoplus_i\omega_{A_i}\longrightarrow\omega_R\longrightarrow \kk^{r-1}\longrightarrow0.
\]
It follows that $a(R)=\max\{0,a(A_1),\ldots,a(A_r)\}$ and that every element of $\omega_R$ of negative degree belongs to $\bigoplus_i\omega_{A_i}$.

\begin{corollary}\label{cor:almost-one}
Let $\kk$ be infinite and $I\subseteq S$ a monomial ideal. Suppose that $R=S/I$ is Cohen--Macaulay of dimension one. Then $R$ is almost Gorenstein if and only if it is a fiber product $A_1\times_{\kk}\cdots\times_{\kk} A_r$ of one-dimensional monomial complete intersections such that either $R$ has minimal multiplicity or $a(A_1)=\cdots=a(A_r)$.
\end{corollary}
\begin{proof}
If $R$ is Gorenstein, the assertion holds with one factor by \cite[Corollary 4.4]{BH95}. Suppose that $R$ is almost Gorenstein and non-Gorenstein. It is then nearly Gorenstein (\cite[Proposition 6.1]{HHS19}), so Theorem~\ref{thm:A}~(ii) expresses $R$ as a fiber product of monomial complete intersections and rings $E_b$.

Each factor is almost Gorenstein after localization at its graded maximal ideal, by \cite[Theorem 1.1]{EGI21} when there is more than one factor. Each factor is level and generically Gorenstein, and is the associated graded ring of its localization. Thus \cite[Theorem 9.8~(2)]{GTT15} shows that it is almost Gorenstein as a graded ring. Since $a(E_b)=b-2>0$, \cite[Lemma 10.2]{GTT15} excludes every factor $E_b$. Hence all factors $A_1,\ldots,A_r$ are complete intersections, and $r\geq2$.

Set $a=a(R)\geq0$. If $a=0$, then $\deg h_R=1$, so $R$ has minimal multiplicity. Suppose that $a>0$, and choose $f\in[\omega_R]_{-a}$ defining the almost Gorenstein embedding. Since $-a<0$, the exact sequence gives $f=(f_i)\in\bigoplus_i\omega_{A_i}$. Each $f_i$ is nonzero, for otherwise $\mathfrak m_{A_i}f=0$, contrary to injectivity. As $f_i$ has degree $-a$, we have $-a\geq-a(A_i)$. Together with $a(A_i)\leq a$, this gives $a(A_i)=a$ for every $i$, proving the necessity.

Conversely, suppose that $R$ is a fiber product as in the statement. A single factor is Gorenstein. If $r\geq2$ and $R$ has minimal multiplicity, then $a(R)=0$; since $R$ is generically Gorenstein, it is almost Gorenstein by \cite[Theorem 10.4]{GTT15}.

It remains to consider $r\geq2$ factors with a common positive $a$-invariant $a$; a nonpositive common $a$-invariant would give $a(R)=0$ and minimal multiplicity. Choose free generators $f_i\in[\omega_{A_i}]_{-a}$ and set $f=(f_i)$. The injection $R\to\bigoplus_i A_i$ shows that multiplication by $f$ defines a degree-zero injection $R\to\omega_R(-a)$.

The exact sequence gives $\mathfrak m_R\omega_R\subseteq\bigoplus_i A_if_i$. Every nonzero homogeneous element of $\mathfrak m_R\omega_R$ has degree greater than $-a$, whereas each $f_i$ has degree $-a$. Its coefficients in the summands $A_if_i$ therefore have positive degree. Hence $\mathfrak m_R\omega_R\subseteq\bigoplus_i\mathfrak m_{A_i}f_i=\mathfrak m_Rf$, proving that the cokernel is annihilated by $\mathfrak m_R$.
\end{proof}

\begin{remark}\label{rem:almost-example}
The fiber product of two copies of $\kk[x,z]/(z^3)$ is almost Gorenstein: both factors are Gorenstein with $a$-invariant one. Its $h$-polynomial is $1+3t+2t^2$, so it does not have minimal multiplicity.
\end{remark}

\section{Dimension two}\label{sec:two}

In dimension two, connectedness also suffices to recover Cohen--Macaulayness. We record the argument separately.

\begin{lemma}\label{lem:connected-dimension-two}
Let $S=\kk[x_1,\ldots,x_n]$, let $I\subseteq S$ be a monomial ideal, and let $\Delta=\{F\subseteq[n]:\mathbf{x}_F\notin\sqrt I\}$ be the simplicial complex defined by $\sqrt I$, where $\mathbf{x}_F=\prod_{i\in F}x_i$. Suppose that $R=S/I$ is unmixed of dimension two. If $\Delta$ is connected and $R$ is locally Gorenstein, then $\Delta$ is a path or a cycle, and $R$ is Cohen--Macaulay.
\end{lemma}

\begin{proof}
For every vertex $i$ of $\Delta$, recall that $R(\{i\})$ is defined in the variables $x_j$, $j\ne i$, by setting $x_i=1$ in the monomial generators of $I$. This ring is one-dimensional and Gorenstein, and its reduced ring is $\kk[\operatorname{lk}_\Delta\{i\}]$. By \cite[Corollary 4.3]{BH95}, the latter is Gorenstein as well. A zero-dimensional complex has a Gorenstein Stanley--Reisner ring only when it has one or two vertices. Hence each vertex of $\Delta$ has degree one or two, so the connected graph $\Delta$ is a path or a cycle.

If $\Delta$ is a cycle, the facets omitting any vertex form a path in its dual graph. Lemma~\ref{lem:components}~(3) gives $I=\varphi_{\mathbf a}(\sqrt I)S$. Since the Stanley--Reisner ring of a cycle is Cohen--Macaulay, so is $R$.

Suppose that $\Delta$ is a path, with edges $F_i=\{i,i+1\}$ for $1\leq i\leq r$. Write $I=\bigcap_{i=1}^r Q_i$ as in Lemma~\ref{lem:components}~(1). We prove by induction on $r$ that any such intersection of pure-power ideals satisfying the adjacent exponent equalities in part (2) is Cohen--Macaulay. If $r=1$, then $I=Q_1$ is generated by pure powers and the assertion is clear. For $r\geq2$, set $J=\bigcap_{i<r}Q_i$. Part (2) of that lemma shows that $x_{r+1}$ has the same exponent $a$ in all $Q_i$ with $i<r$. It also shows that every generator of $Q_{r-1}$ other than $x_{r+1}^a$ belongs to $Q_r$. Consequently,
\[
J+Q_r=Q_{r-1}+Q_r=Q_r+(x_{r+1}^a).
\]
Indeed, $J\subseteq Q_{r-1}$ gives one inclusion, and $x_{r+1}^a\in J$ gives the other. The last ideal is generated by pure powers of all variables except $x_r$, so $S/(J+Q_r)$ is a one-dimensional complete intersection. The truncated path satisfies the same exponent equalities, hence $S/J$ is Cohen--Macaulay of dimension two by induction. Applying the depth lemma to
\[
0\longrightarrow S/I\longrightarrow S/J\oplus S/Q_r\longrightarrow S/(J+Q_r)\longrightarrow0
\]
completes the proof.
\end{proof}

We next consider monomial ideals whose radical defines a path. The next lemma shows that membership of a variable in the canonical trace ideal forces the corresponding exponents in the primary decomposition to be one.

\begin{lemma}\label{lem:path-trace}
Let $I\subseteq S$ be a monomial ideal such that $R=S/I$ is Cohen--Macaulay and locally Gorenstein. Suppose that $\sqrt I=I_\Delta S+(x_j:j\notin V)$, where $\Delta$ is a path on $V\subseteq[n]$ with $|V|\geq4$, and let $E$ be its edge set. If $x_i\in\operatorname{tr}_R(\omega_R)$, then $x_i/1=0$ in $R_{\mathbf{x}_F}$ for every edge $F$ omitting $i$. In particular, if $R$ is nearly Gorenstein, then $I=\sqrt I$.
\end{lemma}

\begin{proof}
By Lemma~\ref{lem:components}, write $I=\bigcap_{F\in E}(x_l^{a_{F,l}}:l\notin F)$. Number the vertices of the path consecutively as $1,\ldots,m$ and set $F_0=\{1,m\}$. Choose $N\geq a_{F,l}$ for all $F,l$, and set
\[
J=\bigcap_{F\in E\cup\{F_0\}}(x_l^N:l\notin F).
\]
The edge $F_0$ closes the path to a cycle. We have $J\subseteq I$, and $B=S/J$ is obtained from the Stanley--Reisner ring of this cycle by substitution by powers. The ring of a cycle is Gorenstein (see \cite[Corollary 5.6.5]{BH98}), so $B$ is Gorenstein as well. Since $\dim B=\dim R$, \cite[Proposition 3.6.12~(b)]{BH98} identifies $C=(J:I)/J$ with $\omega_R$ up to a multigraded shift.

Each edge of the path omits at least one endpoint. By the choice of $N$, the monomial $\mathbf{x}_{F_0}^N$ therefore belongs to every primary component of $I$, and hence to $I$. It follows that
\[
J:I\subseteq J:\mathbf{x}_{F_0}^N=(x_j^N:j\notin F_0).
\]
Suppose that $x_i\in\operatorname{tr}_R(\omega_R)$. The assertion is clear if $x_i=0$. Otherwise choose a homogeneous map $\varphi:C\to R$ and a monomial $\mathbf{w}\in C$ with $\varphi(\mathbf{w})=x_i$. The displayed containment gives $x_j^N\mid \mathbf{w}$ for some $j$. Since $JS_{\mathbf{x}_F}=(x_l^N:l\notin F)S_{\mathbf{x}_F}$, we have $\mathbf{w}/1=0$ in $C_{\mathbf{x}_F}$ whenever $j\notin F$. Applying $\varphi$, we obtain $x_i/1=0$ in $R_{\mathbf{x}_F}$ for every edge $F$ omitting $j$.

It follows that every edge containing $i$ also contains $j$, because $x_i$ is invertible after localizing at such an edge. If $j=i$, this proves the assertion. Suppose that $j\ne i$. Then at most one edge contains $i$, so the edges omitting $i$ form a connected subgraph of the dual graph. Since $|V|\geq4$, some edge $F$ omits $j$, and hence also $i$. We have $x_i/1=0$ in $R_{\mathbf{x}_F}$, or equivalently $a_{F,i}=1$. By Lemma~\ref{lem:components}~(2), the exponents $a_{F,i}$ agree along the connected subgraph of edges omitting $i$. They are therefore all equal to one.

If $R$ is nearly Gorenstein, the conclusion applies to every variable. All exponents in the primary decomposition are then one, so $I=\sqrt I$.
\end{proof}

\begin{proof}[Proof of Theorem~\ref{thm:A}~(iii)]
We first prove the sufficiency. If $I$ has the stated form, then $R$ is the Stanley--Reisner ring of a path, which is nearly Gorenstein by \cite[Theorem 4.3~(b)]{M24}.

We now prove the necessity. Suppose that $R$ is nearly Gorenstein. In view of Lemma~\ref{lem:path-trace}, it suffices to show that $\sqrt I$ defines a path with at least four vertices.

Write $S/\sqrt I=\kk[\Delta]$, where $\Delta$ has vertex set $V\subseteq[n]$. Since $R$ is Cohen--Macaulay, $\Delta$ is connected. Lemma~\ref{lem:connected-dimension-two} shows that it is a path or a cycle. It has no cone point by Lemma~\ref{lem:components}~(4), so a path must have at least four vertices.

Suppose that $\Delta$ is a cycle. The edges omitting any $i\in[n]$ induce a connected subgraph of the dual graph, so Lemma~\ref{lem:components}~(3) gives $I=\varphi_{\mathbf a}(\sqrt I)S$. Since the Stanley--Reisner ring of a cycle is Gorenstein (\cite[Corollary 5.6.5]{BH98}), flat base change of the canonical module would make $R$ Gorenstein as well, a contradiction.

Consequently $\Delta$ is a path with at least four vertices. Lemma~\ref{lem:path-trace} now gives $I=\sqrt I$. Since $I\subseteq(x_1,\ldots,x_n)^2$, we have $V=[n]$, and numbering the vertices consecutively yields the stated ideal.
\end{proof}

\section{The dimension bound}\label{sec:dimension-bound}

We prove Theorem~\ref{thm:B}, which supplies the remaining dimension bound in Theorem~\ref{thm:A}. We first reduce the case of dimension at least three to a Stanley--Reisner ring.

\begin{lemma}\label{lem:manifold}
Let $S=\kk[x_1,\ldots,x_n]$, let $I\subseteq S$ be a monomial ideal, and let $\Delta=\{F\subseteq[n]:\prod_{i\in F}x_i\notin\sqrt I\}$ be the simplicial complex defined by $\sqrt I$. Suppose that $R=S/I$ is unmixed of dimension $d\geq3$ and that $\Delta$ is connected. Assume that $R$ is not quasi-Gorenstein and that $R_{\mathfrak p}$ is quasi-Gorenstein for every prime $\mathfrak p\ne\mathfrak m_R$. Then $I=\varphi_{\mathbf a}(\sqrt I)S$ for some $a_i\geq1$, where $\varphi_{\mathbf a}:S\to S$ sends $x_i$ to $x_i^{a_i}$.
\end{lemma}

\begin{proof}
The complex $\Delta$ is pure, since $R$ is unmixed. For a face $C$, recall that $R(C)$ is defined in the variables $x_j$, $j\notin C$, by setting $x_i=1$ for $i\in C$ in the monomial generators of $I$; its reduced ring is $\kk[\operatorname{lk}_\Delta C]$.

We first check that every link of positive dimension is connected. For a nonempty face $C$ with $\dim\operatorname{lk}_\Delta C\geq1$, the ring $R(C)$ is quasi-Gorenstein of dimension at least two by the proof of Lemma~\ref{lem:components}. Hence $H^1_{\mathfrak m_{R(C)}}(R(C))=0$. By \cite[Corollary 2.3]{HTT05}, passing to the radical does not change the multidegree-zero component of this module. Hochster's formula (\cite[Theorem 5.3.8]{BH98}) gives its dimension as $\dim_{\kk}\widetilde H_0(\operatorname{lk}_\Delta C;\kk)$. The link is therefore connected. For the empty face, connectedness is our assumption on $\Delta$. It follows that the dual graph of $\Delta$ is connected (see \cite[Remark 4.6]{KM26}).

There is no cone point by Lemma~\ref{lem:components}~(4). If $C$ is a nonempty face whose link has dimension at most one, then $R(C)$ is Gorenstein by the proof of that lemma. By \cite[Corollary 4.3]{BH95}, its reduced ring $\kk[\operatorname{lk}_\Delta C]$ is Gorenstein as well. Hence \cite[Proposition 4.13]{KM26} shows that $\Delta$ is a \textit{pseudomanifold}: its dual graph is connected and each codimension-one face belongs to exactly two facets.

For each vertex $j$ of $\Delta$, the facets omitting $j$ induce a connected subgraph by \cite[Lemma 2.1]{A11}. If $x_j\in\sqrt I$, every facet omits $j$, so the corresponding subgraph is the whole dual graph. Lemma~\ref{lem:components}~(3) now gives the required exponents.
\end{proof}

\begin{proof}[Proof of Theorem~\ref{thm:B}]
The equality $\sqrt{\operatorname{tr}_R(\omega_R)}=\mathfrak m_R$ implies that $R$ is not quasi-Gorenstein and that $R_{\mathfrak p}$ is quasi-Gorenstein for every $\mathfrak p\ne\mathfrak m_R$ (\cite[Remark 2.9~(4)--(6)]{KM26}). If $\dim R\leq1$, unmixedness implies Cohen--Macaulayness. If $\dim R=2$, the localizations away from $\mathfrak m_R$ are Gorenstein, and Lemma~\ref{lem:connected-dimension-two} applies.

Suppose, to the contrary, that $\dim R\geq3$. By Lemma~\ref{lem:manifold}, we have $I=\varphi_{\mathbf a}(\sqrt I)S$. Set $B=S/\sqrt I$. The map $B\to R$ sending $x_i$ to $x_i^{a_i}$ is faithfully flat, and \eqref{eq:power-trace} gives $\operatorname{tr}_R(\omega_R)=\varphi_{\mathbf a}(\operatorname{tr}_B(\omega_B))R$.

We also have $\sqrt{\operatorname{tr}_B(\omega_B)}=\mathfrak m_B$. Indeed, faithful flatness implies that an extended ideal contracts to the original ideal. Moreover, $\mathfrak m_R$ contracts to $\mathfrak m_B$, since the map sends each variable to a positive power and induces the identity on $\kk$. The equality follows by contracting $\sqrt{\operatorname{tr}_R(\omega_R)}=\mathfrak m_R$.

The ring $B$ is the Stanley--Reisner ring of the connected complex $\Delta$. It is not quasi-Gorenstein, and its localizations away from $\mathfrak m_B$ are quasi-Gorenstein. Since $\dim B\geq3$, \cite[Theorem 4.16]{KM26}, applied with the standard grading, gives $\operatorname{tr}_B(\omega_B)\subseteq\mathfrak m_B^2$. Every product of two variables maps to a product $x_i^{a_i}x_j^{a_j}\in\mathfrak m_R^2$, because $a_i,a_j\geq1$. Hence
\[
\operatorname{tr}_R(\omega_R)=\varphi_{\mathbf a}\bigl(\operatorname{tr}_B(\omega_B)\bigr)R\subseteq\varphi_{\mathbf a}(\mathfrak m_B^2)R\subseteq\mathfrak m_R^2,
\]
contrary to the second assumption.
\end{proof}

The conclusion also holds under Serre's condition $(S_2)$. Recall that $R$ satisfies \textit{Serre's condition $(S_2)$} if $\operatorname{depth}R_{\mathfrak p}\geq\min\{2,\dim R_{\mathfrak p}\}$ for every prime $\mathfrak p$.

\begin{corollary}\label{cor:s2-bound}
Let $R=S/I$ satisfy $(S_2)$, where $I$ is a monomial ideal. If $\sqrt{\operatorname{tr}_R(\omega_R)}=\mathfrak m_R$ and $\operatorname{tr}_R(\omega_R)\nsubseteq\mathfrak m_R^2$, then $R$ is Cohen--Macaulay and $\dim R\leq2$.
\end{corollary}

\begin{proof}
Let $\Delta$ be the simplicial complex defined by $\sqrt I$. The condition $(S_2)$ implies that $R$ is unmixed (see \cite[Remark 2.9~(4)]{CV20}). If $\dim R\geq2$, Hartshorne's connectedness theorem gives connectedness of its dual graph, and hence of $\Delta$ (see \cite[Proposition 2.2.6]{BCRV22}). Theorem~\ref{thm:B} applies.
\end{proof}

\begin{proof}[Completion of the proof of Theorem~\ref{thm:A}]
Sections~\ref{sec:artinian}, \ref{sec:one}, and~\ref{sec:two} establish the classification in dimensions zero, one, and two, respectively. If $R$ is non-Gorenstein and nearly Gorenstein, then it is Cohen--Macaulay, hence satisfies $(S_2)$, and $\operatorname{tr}_R(\omega_R)=\mathfrak m_R$. Corollary~\ref{cor:s2-bound} gives $\dim R\leq2$, completing the proof.
\end{proof}

\begin{remark}
Neither assumption on the canonical trace ideal in Theorem~\ref{thm:B} can be omitted. For a ring $B$ in Theorem~\ref{thm:A}~(iii), the polynomial ring $B[t]$ is Cohen--Macaulay of dimension three and has canonical trace ideal $\mathfrak m_BB[t]$ by \cite[Theorem 4.2]{HHS19}. This ideal contains variables but is not primary to the graded maximal ideal. In characteristic different from two, the Stanley--Reisner ring of a triangulation of the real projective plane is Cohen--Macaulay of dimension three and has canonical trace ideal equal to the square of its graded maximal ideal (see \cite[Remark 3.8 and Corollary 3.10]{MV25}).
\end{remark}

The same reduction gives a further consequence in even dimension.

\begin{corollary}\label{cor:even}
Let $I\subseteq S$ be a monomial ideal, and suppose that $R=S/I$ is Cohen--Macaulay of even dimension $d\geq4$. If $R$ is locally Gorenstein, then it is Gorenstein.
\end{corollary}

\begin{proof}
Let $\Delta$ be the simplicial complex defined by $\sqrt I$, and suppose that $R$ is not Gorenstein. Since $R$ is Cohen--Macaulay, it is unmixed and $\Delta$ is connected, as in the proof of Corollary~\ref{cor:s2-bound}. By Lemma~\ref{lem:manifold} and faithfully flat descent, $B=S/\sqrt I$ is a Cohen--Macaulay, locally Gorenstein, non-Gorenstein Stanley--Reisner ring of dimension $d$. Combining \cite[Theorem A]{MV25} with Hochster's formula (see \cite[Theorem 5.3.8]{BH98}) and graded local duality, we find that $\omega_B$ and $\mathfrak m_B$ have the same multigraded Hilbert function. Writing $H_M(t)$ for the Hilbert series of $M$, we obtain $H_{\omega_B}(t)=H_B(t)-1$. On the other hand, since $d$ is even, \cite[Theorem 4.4.5~(a)]{BH98} gives $H_{\omega_B}(t)=H_B(t^{-1})$. Both sides are defined at $t=-1$. Substitution in the two identities gives $H_B(-1)=H_B(-1)-1$, a contradiction.
\end{proof}

\begin{remark}
The ring $\kk[x_1,x_2,x_3,x_4]/(x_1x_3,x_1x_4,x_2x_4)$ is Cohen--Macaulay and locally Gorenstein, but not Gorenstein. Thus the conclusion of Corollary~\ref{cor:even} fails in dimension two.
\end{remark}

\begin{corollary}\label{cor:radical}
The radical of a Cohen--Macaulay nearly Gorenstein monomial ideal is nearly Gorenstein.
\end{corollary}
\begin{proof}
If $S/I$ is Gorenstein, then $S/\sqrt I$ is Gorenstein by \cite[Corollary 4.3]{BH95}. Otherwise Theorem~\ref{thm:A} gives $\dim S/I\leq2$. In dimension zero, $S/\sqrt I=\kk$. In dimension one, $S/\sqrt I$ is the Stanley--Reisner ring of a zero-dimensional complex, hence a fiber product of polynomial rings in one variable. It is nearly Gorenstein by \cite[Theorem 5.2]{EGI21}. In dimension two, $I$ is already radical.
\end{proof}

\subsection*{Acknowledgments}
This paper is a thoroughly revised and expanded version of the author's earlier preprint, ``Canonical traces of Artinian truncations of Stanley--Reisner rings'' (\href{https://arxiv.org/abs/2608.15955}{arXiv:2608.15955}).

The author thanks Matteo Varbaro and Naoki Terai for helpful discussions and acknowledges support from JSPS KAKENHI Grant JP25KJ1744. ChatGPT (OpenAI), including ChatGPT 6.1 Sol, assisted with proofs, references, and writing. The author is responsible for all content.

\end{document}